\documentclass[a4paper,11pt,leqno,english]{smfart}
\usepackage{aeguill}
\usepackage{enumerate}
\usepackage{amssymb,amsmath,latexsym,amsthm}
\usepackage[T1]{fontenc}
\usepackage{smfthm}
\usepackage{geometry}   
\usepackage{url}      
\usepackage[frenchb, english]{babel}
\usepackage[utf8]{inputenc}   
\usepackage{mathrsfs}
\usepackage{relsize}
\usepackage{xcolor}
\usepackage[all]{xy}
\usepackage{comment}
\definecolor{violet}{rgb}{0.0,0.2,0.7}
\definecolor{rouge2}{rgb}{0.8,0.0,0.2}
\usepackage{hyperref}
\usepackage{mathrsfs}
\hypersetup{
    bookmarks=true,         % show bookmarks bar?
    unicode=false,          % non-Latin characters in Acrobat’s bookmarks
    pdftoolbar=true,        % show Acrobat’s toolbar?
    pdfmenubar=true,        % show Acrobat’s menu?
    pdffitwindow=false,     % window fit to page when opened
    pdfstartview={FitH},    % fits the width of the page to the window
    pdftitle={},    % title
    pdfauthor={},     % author
    colorlinks=true,       % false: boxed links; true: colored links
   linkcolor=violet,          % color of internal links
    citecolor=violet,        % color of links to bibliography
    filecolor=black,      % color of file links
    urlcolor=cyan}           % color of external links
\newcommand{\X}{X^{\circ}}

\newcommand{\Q}{\mathbb{Q}}

\renewcommand{\O}{\mathcal{O}}

\newcommand{\ep}{\varepsilon}

\newcommand{\la}{\langle}

\newcommand{\ra}{\rangle}

\renewcommand{\ge}{\geqslant}
\renewcommand{\le}{\leqslant}

\newcommand{\Ric}{\mathrm{Ric} \,}
\newcommand{\om}{\omega}

\newcommand{\omke}{\omega_{\rm KE}}
\newcommand{\omkeb}{\omega_{{\rm KE},b}}
\newcommand{\omkebs}{\omega_{{\rm KE},b,s}}

\newcommand{\omt}{\om_{t}}

\newcommand{\Supp}{\mathrm {Supp}}
\newcommand{\omp}{\omega_{\rm P}}

\newcommand{\Db}{\Delta_{b,s}}

\newcommand{\Zsd}{Z^{\circ}}

\newcommand{\tr}{\mathrm{tr}}

\newcommand{\ddc}{dd^c}
\newcommand{\vpb}{\varphi_{b}}
\newcommand{\vpte}{\varphi_{t,\ep}}
\newcommand{\omte}{\omega_{t,\ep}}
\newcommand{\omted}{\omega_{t,\ep,\delta}}

\newcommand{\tom}{\widetilde{\omega}}

\newcommand{\omb}{\omega_{b}}

\newcommand{\omz}{\omega_{Z}}

\renewcommand{\ae}{\alpha_{\ep}}
\newcommand{\be}{\beta_{\ep}}

\def\HSC{\mathrm{HSC}}

\newtheorem*{thma}{Theorem A}
\newtheorem*{thmb}{Theorem B}

\numberwithin{equation}{section}

\begin{document}

\frontmatter 

\title[Quasi-projective manifolds with negative HSC]{Quasi-projective manifolds with negative holomorphic sectional curvature}

\author{Henri Guenancia}
\address{Institut de Mathématiques de Toulouse; UMR 5219, Université de Toulouse; CNRS, UPS, 118 route de Narbonne, F-31062 Toulouse Cedex 9, France}
\email{henri.guenancia@math.univ-toulouse.fr}
\date{\today}
\thanks{The author is partially supported by the NSF Grant DMS-1510214 and the project PEPS "Jeune chercheuse, jeune chercheur" funded by the CNRS}

\begin{abstract}
Let $(M,\omega)$ be a compact Kähler manifold with negative holomorphic sectional curvature. It was proved by Wu-Yau and Tosatti-Yang that $M$ is necessarily projective and has ample canonical bundle. In this paper, we show that any irreducible subvariety of $M$ is of general type. Moreover, we can extend the theorem to the quasi-negative curvature case building on earlier results of Diverio-Trapani. Finally, we investigate the more general setting of a quasi-projective manifold $X^{\circ}$ endowed with a Kähler metric with negative holomorphic sectional curvature and we prove that such a manifold $X^{\circ}$ is necessarily of log general type.  
\end{abstract}

%\begin{center}
%{\Large
%\textsc{Singular subvarieties of Kähler manifolds with negative holomorphic sectional curvature}}
%\end{center}
\maketitle
%\tableofcontents
\bigskip

\section{Introduction}

\subsection{Singular subvarieties}
Let $M$ be a compact Kähler manifold of dimension $n$ and let $\omega$ be a Kähler metric on $M$ such that its holomorphic sectional curvature is negative; that is, for every $x\in M$ and any $[v]\in \mathbb P(T_{M,x})$, one has $\HSC_{\om}(x,[v])<0$. 

%\textit{quasi-negative}, that is
%\begin{enumerate}
%\item[$(i)$] For every $x\in X$ and any $v\in T_xX$, one has $\HSC(X,\om)_x (v) \le 0$.
%\item[$(ii)$] There exists $x_0\in X$ such that for any $v\in T_{x_0}X\setminus\{0\}$, $\HSC(X,\om)_{x_0} (v) < 0$.
%\end{enumerate} 

Recall that if $(R_{i\bar j k \bar \ell })$ is the curvature tensor of $\om$ in some holomorphic coordinates $(z_i)$ and if $v=\sum v_i \frac{\partial}{\partial z_i}$ is a non-zero tangent vector at $x$, then the holomorphic sectional curvature of $(M,\om)$ at $(x,[v])$ is defined by 
$$\HSC_{\om}(x,[v]):=\frac{1}{|v|^4_{\om}} \cdotp \sum_{i,j,k,\ell } R_{i\bar j k \bar \ell} \, v_i\bar v_jv_k\bar v_{\ell}.$$ 
%Let us define the subset $\Omega\subset X$ by $\Omega:=\{x\in X; \forall v\in T_xX\setminus \{0\}, \HSC(X,\om)_x (v) <0\}$. By the condition $(ii)$ above, $\Omega$ is a non-empty subset of $X$, open for the euclidean topology. 
Under the assumptions on $(M,\om)$ above, it was proved by Wu-Yau \cite{WY16} that $K_M$ is ample provided that $M$ is a projective manifold. Shortly after, Tosatti-Yang \cite{TosYang} extended the result to the general Kähler case. In particular, under those general assumptions, $M$ is automatically projective. Now, if $Y\subset M$ is a \textit{smooth} subvariety of $M$, then the decreasing property of the holomorphic (bi)sectional curvature shows that $K_Y$ is ample again. However, this argument cannot be directly applied to \textit{singular} subvarieties of $M$. The first main result of this paper deals precisely with this question. 

\begin{thma}
Let $(M,\om)$ be a compact Kähler manifold with negative holomorphic sectional curvature and let $Y\subset M$ be a possibly singular, irreducible subvariety of dimension $m$. Then, $Y$ is of general type.  
\end{thma}

Recall that a smooth projective variety $Y$ of dimension $m$ is said to be of general type if there exists a constant $C>0$ such that $$h^0(Y,K_Y^{\otimes k}) \ge C k^m$$ for any $k$ large enough. Moreover, an irreducible projective variety $Y$ is said to be of general type if there exists a resolution $\widetilde Y \to Y^n$ of the normalisation $Y^n$ of $Y$ such that $\widetilde Y$ is of general type. By invariance of the Kodaira dimension of projective manifolds under birational morphisms, this condition is independent of the resolution.\\

The strategy and main ideas behind the proof of the result above are outlined in Section~\ref{proof} below Corollary~\ref{coro}.\\

\noindent
\textit{The quasi-negative curvature case.} Theorem A generalizes to the case of quasi-negative holomorphic sectional curvature, where one needs to use as an important first step a result of Diverio-Trapani \cite{DT16}. We refer to \S~\!\ref{quasin} and Theorem~\ref{thmqn} for a statement and a proof. \\

\noindent
\textit{Log terminal subvarieties. } In the setting of the Theorem A, one can additionally show that if $Y$ has log terminal singularities, then $K_Y$ is an ample $\Q$-line bundle, cf Remark~\ref{lt}.  \\

\subsection{The general quasi-projective case}

Another way to think of the situation of Theorem~A is to view $Y_{\rm reg}$ as a quasi-projective manifold endowed with a Kähler metric $\om$ such that
\begin{enumerate}
\item $\om$ has negative holomorphic sectional curvature;
\item $\om$ extends smoothly to a (singular) compactification.
\end{enumerate}
Given this point of view, it is natural to ask to which extent Theorem~A generalizes to arbitrary quasi-projective manifolds. More precisely, given a projective manifold $X$, a reduced divisor $D$ with simple normal crossings and a Kähler metric $\omega$ on $\X:=X\backslash D$ with negative holomorphic curvature, is it true that $(X,D)$ is of log general type; that is, $K_X+D$ is big?

This question is in part motivated by recent results of Cadorel \cite{Cadorel} who proved that given a projective log smooth pair $(X,D)$ such that $\X$ admits a Kähler metric $\om$ with negative holomorphic sectional curvature and non-positive holomorphic \textit{bisectional} curvature, then $\Omega_X(\log D)$ is big, and, moreover, $\Omega_X$ is big provided that $\om$ is bounded near $D$. 

His proof involves  working on $\mathbb P(\Omega_X(\log D))$ and considering the tautological line bundle $\mathcal O(1)$ on it. By the assumption on the \textit{bisectional} curvature, $\om$ induces a smooth, non-negatively curved hermitian metric $h$ on $\mathcal O(1)$ away from (the inverse image of) $D$. Moreover, the Alhfors-Schwarz lemma guarantees that $h$ extends across $D$ as a singular metric with non-negative curvature. Using a result of Boucksom \cite{Bou02} on a metric characterization of bigness then completes the proof. \\

One cannot expect such a strong result on the logarithmic cotangent bundle if one drops the assumption on the bisectional curvature. However, it seems reasonable to expect it for the logarithmic canonical bundle. The main difficulty is that one does not get from $\om$ a positively curved metric on $K_X+D$ even on a Zariski open set. So one has to produce such a metric out of other methods, like the continuity method, cf \cite{WY16}. However, one faces several new difficulties compared to the setting of Theorem A: 

\begin{enumerate}
\item To start the continuity method, one needs $K_X+D$ to be pseudo-effective. In the case $D=0$, this is a consequence of the absence of rational curves (Ahlfors-Schwarz lemma) combined with Mori's bend and break and \cite{BDPP}. If $D$ is not empty then one only knows that $\X$ has no entire curves hence $X$ has no rational curve meeting $D$ at at most two points. To conclude, one would then need to have a logarithmic version of Mori's bend and break, but unfortunately it is not known as of now, cf Remark~\ref{LBB}. To circumvent the difficulty and inspired by the proof of \cite[Thm.~4.1]{CP}, we modify the boundary $D$ into $D+sB$ for some ample $B$ and some $s>0$ to make $K_X+D+sB$ psef. Only at the very end of the argument, one will see that $K_X+D$ is pseudoeffective. 

\item The finiteness of the log canonical ring, known for klt pairs and crucial to understanding the deforming Kähler-Einstein metrics, is not known for lc pairs like $(X,D)$. The idea is then to deform $(X,D)$ into a klt pair $(X,\Db:=(1-b)D+(b+s)B)$ that makes it klt and of log general type. The price to pay is that we have to carry on an additional error term in the volume estimate (compare Proposition~\ref{intcontrol} and Theorem~\ref{intcontrol2}). 
\end{enumerate}

Give or take these adjustements, one can still run the strategy of Theorem~A \textit{mutatis mutandis}; it will tell us that the volume of $K_X+(1-b)D+(b+s)B$ is bounded away from zero uniformly in $b,s>0$. One of the main points, already present in \cite{Cadorel}, is that the behavior of $\om$ near $D$ is not arbitrary, as $\om$ must be dominated by a metric with Poincaré singularities along $D$ thanks to Ahlfors-Schwarz lemma. However, one needs to look early on at $\om$ on birational models of $(X,D)$ where the Kähler-Einstein metrics are better understood, and $\om$ will pick up singularities along exceptional divisors which will complicate the argument. In the end, the result is the following

\begin{thmb}
Let $(X,D)$ be a pair consisting of a projective manifold $X$ and a reduced divisor $D=\sum_{i\in I} D_i$ with simple normal crossings. Let $\omega$ be a Kähler metric on $\X:=X\backslash D$ such that there exists $\kappa_0>0$ satisfying 
$$\forall (x,v)\in \X \times T_{X,x}\backslash \{0\}, \quad \HSC_{\om}(x,[v])<-\kappa_0.$$
Then, the pair $(X,D)$ is of log general type; that is, $K_X+D$ is big. If additionnally $\om$ is assumed to be bounded near $D$, then $K_X$ is big.
\end{thmb}

\noindent
In particular, Theorem~A is a corollary of Theorem~B. However, we chose to state and prove Theorem A separately in order to better highlight the new ideas that are necessary for Theorem A (and its quasi-negative analogue) and then only later add a layer of technicality to go from Theorem A to the more general Theorem B.

\section*{Acknowledgements}
I would like to thank Simone Diverio for introducing me to this problem and for the many related insightful discussions. I am very much indebted to Sébastien Boucksom for his comments on a preliminary draft of this paper and for suggesting me to consider the quasi-projective case. Finally, I am grateful to Benoît Cadorel for interesting exchanges about this topic.

\section{Proof of Theorem~A}
\label{proof}
Let $(M,\om)$ as in the Theorem, and let $Y\subset M$ be an irreducible subvariety of dimension $m$. As we recalled, $M$ is necessarily projective and does not contain any rational curve; that is, any holomorphic map $\mathbb P^1 \to M$ is constant, cf \cite[Cor.~2]{Royden}. One considers $p:X\to Y$ a resolution of the normalization of $Y$, and the goal is to show that $X$ is of general type using the special Kähler metric $\om|_{Y}$. 

 The first observation is that $X$ cannot be uniruled. Otherwise, so would be $Y$, which is prevented by the non-existence of rational curves on $M$. Therefore, the fundamental result of \cite{BDPP} guarantees that $K_X$ is pseudo-effective. 
 
 The second important observation is that $p^*(\om|_{Y})$ is a smooth closed $(1,1)$-form on $X$ which is positive on a Zariski open set $\Omega$ of $X$. Moreover, there exists $\kappa_0>0$ such that the Kähler metric $(p^*(\om|_{Y}))|_{\Omega} $ has holomorphic sectional curvature bounded above by $-\kappa_0$. This is because the holomorphic sectional curvature of the Kähler metric $\om|_{Y_{\rm reg}}$ admits such a bound by the compactness of $M$ and the decreasing property of the bisectional curvature. These observations lead us to consider the following setting. 
 
 \subsection{Setting}
 \label{setting}
 Let $X$ be a smooth, complex projective variety of dimension $m$. Let $\om$ be a smooth, closed, semipositive $(1,1)$-form on $X$ such that there exists a Zariski open subset $\Omega \subset X$ satisfying: 
\begin{enumerate}
\item The restriction $\om|_{\Omega}$ is a Kähler metric on $\Omega$.
\item There exists $\kappa_0 >0$ such that for any $(x,[v])\in \Omega\times \mathbb P(T_{X,x})$, one has $$\HSC_{\om}(x,[v])\le -\kappa_0.$$
\end{enumerate}
Moreover, let $B$ be a smooth divisor such that $K_X+bB$ is a big $\Q$-divisor for some rational number $b\in [0,1)$. Let $\omkeb$ be the Kähler-Einstein metric associated to the pair $(X,bB)$. That is, $\omkeb$ is a closed, positive current with minimal singularities in $c_1(K_X+bB)$ satisfying the Einstein equation 
$$\Ric \omkeb = -\omkeb +b[B]$$ 
cf \cite{BEGZ}. That current defines a smooth Kähler metric on the Zariski open set $\mathrm{Amp}(K_X+bB)\backslash B$ thanks to the techniques of \textit{loc. cit.} (cf. also \cite{G2}) and the existence of a log canonical model for $(X,bB)$, cf \cite{BCHM}.  \\

\noindent
The following proposition is the crucial estimate needed for the proof of the main Theorem. 
\begin{prop}
\label{intcontrol}
In the setting \ref{setting} above, there exists a constant $C=C(m,\kappa_0)$ independent of $b$ such that 
$$\int_{\mathrm{Amp}(K_X+bB)\backslash B} \tr_{\omkeb}\om \, \cdotp  \omkeb^m\,  \le C \, \mathrm{vol}(K_X+bB).$$
\end{prop}

\noindent
Using the proposition above, Theorem~A follows relatively quickly.

\begin{coro}
\label{coro}
In the setting \ref{setting} above, $X$ is of general type, ie $K_X$ is big.
\end{coro}
The idea of the proof of the Corollary is to consider an ample divisor $B$ on $X$ and analyze the family of singular Kähler-Einstein metrics $\om_{{\rm KE},b}$ of the pairs of log general type $(X,bB)$ when $b>0$ approaches zero. More precisely, the main point is to show that the volume of these singular metrics does not go to zero when $b\to 0$. The metrics $\omkeb$  are not so well understood directly on $X$, but become much more manageable when seen on the log canonical model $X_{{\rm can},b}$ of the pair $(X,bB)$ whose existence is guaranteed by the fundamental results of \cite{BCHM}. However, these models vary with $b$, hence it is crucial that the estimates be obtained on the fixed manifold $X$, which is the essence of Proposition~\ref{intcontrol}. 

\subsection{Proof of the volume estimate}
This section is devoted to the proof of Proposition~\ref{intcontrol}.
\begin{proof}[Proof of Proposition~\ref{intcontrol}]
By \cite{BCHM}, there exists a canonical model $(X_{\rm can},bB_{\rm can})$ of $(X,B)$ with klt singularities such that $K_{X_{can}}+bB_{\rm can}$ is ample. Let us consider a resolution $Z$ of the graph of the birational map $\phi:X\dashrightarrow X_{\rm can}$ as summarized in the diagram below
 $$
  \xymatrix{
    & \ar[ld]_{\mu} Z  \ar[rd]^{\nu}   &\\
      X  \ar@{-->}[rr]_{\phi} &  &   X_{\rm can} 
  }
  $$
  Then, there exists a $\mathbb Q$-divisor $B_Z=\sum_{i=0}^r b_i B_i$ with snc support, coefficients $b_i\in (0,1)$, with $b_0=b$, $\mu_*B_0=B$ and $B_i$ being $\nu$-exceptional for $i=1, \ldots , r$ such that 
  $$K_Z+B_Z=\nu^* (K_{X_{can}}+bB_{\rm can})+E_Z$$
  for some effective, $\nu$-exceptional $\Q$-divisor $E_Z=\sum_{j=0}^d a_j E_j$. Let us stress here that $\mu$ is an isomorphism over the Zariski open set $\mathrm{Amp}(K_X+bB)$ given that $\phi$ is defined there and induces an isomorphism onto its image when restricted to that set. \\
  
  Let $A:=K_{X_{can}}+bB_{\rm can}$ and let $\om_Z$ be a background Kähler metric on $Z$. For any $t\in [0,1]$, the cohomology class $c_1(\nu^*A+t\{\om_Z\})$ is semi-positive and big (it is even Kähler if $t>0$). Thus,  it follows from \cite{EGZ} that there exists a unique singular Kähler-Einstein metric $\om_t\in c_1(\nu^*A+t\{\om_Z\})$ solving 
  $$\Ric \omt = -\omt + t\om_Z+ [B_Z]-[E_Z]$$
  Moreover, the current $\omt$ has bounded potentials for any $t\in [0,1]$ and there exists an effective, $\mu$-exceptional $\Q$-divisor $F$ on $Z$ such that 
  \begin{equation}
  \label{keb}
  \mu^*\omkeb = \om_0+[F].
  \end{equation} 
  \bigskip

  \noindent
  \textbf{Step 1. Approximate KE metrics on a birational model}
  
  \noindent
  In the following, we will introduce a family of smooth approximations $(\omte)_{\ep>0}$ of $\omt$ defined as follows. Let us choose on $\mathcal O_Z(B_i)$ (resp. $\mathcal O_Z(E_j)$) a holomorphic section $s_i$ (resp. $t_j$) cutting out $B_i$ (resp. $E_j$) and a smooth hermitian metric $h_{B_i}$ (resp. $h_{E_j}$) with Chern curvature $\Theta_{h_{B_i}}$ (resp. $\Theta_{h_{E_j}}$). In order to lighten notation, one sets $|s_i|^2:=h_{B_i}(s_i,s_i)$ (resp. $|t_j|^2:=h_{E_j}(t_j,t_j)$). For any $\ep \in (0,1)$, one defines $\theta_{\ep}^B:=\sum_{i=0}^r b_i (\Theta_{h_{B_i}}+dd^c \log(|s_i|^2+\ep^2))$ and similarly $\theta_{\ep}^E:=\sum_{j=0}^d a_j (\Theta_{h_{E_j}}+dd^c \log(|t_j|^2+\ep^2))$. The smooth $(1,1)$-form $\theta_{\ep}^B$ represents $c_1(B_Z)$ and converges weakly to the current of integration $[B_Z]$ when $\ep\to 0$, and similarly for $\theta_{\ep}^{E}$. Thanks to \cite{Aubin,Yau78}, there exists for any $t,\ep>0$ a unique smooth, Kähler metric $\omte \in c_1(\nu^*A+t\{\om_Z\})$ such that 
  \begin{equation}
  \label{ric}
  \Ric \omte = -\omte + t\om_Z+ \theta_{\ep}^B-\theta_{\ep}^E
  \end{equation}
  In terms of Monge-Ampère equations, this is equivalent to saying that $\omte=\nu^*\om_A+t\om_Z+dd^c \vpte$ solves
  $$(\nu^*\om_A+t\om_Z+dd^c \vpte)^m=\frac{\prod_{j=0}^d (|t_j|^2+\ep^2)^{a_j}}{\prod_{i=0}^d (|s_i|^2+\ep^2)^{b_i}}e^{\vpte} dV$$
  where $\om_A\in c_1(A)$ is a Kähler form on $X_{\rm can}$ and $dV$ is a smooth volume form chosen such that $\Ric dV=-\nu^*\om_A+t\om_Z+\sum b_i \Theta_{h_{B_i}}- \sum a_j \Theta_{h_{E_j}}$. By the proof of \cite[Prop.~1]{GP} and the estimates of \cite[Sect.~4]{GP}, there exists a constant $C_t$ independent of $\ep>0$ such that 
  \begin{equation}
  \label{conic}
  \omte \le C_t\, \om_{B_Z,\ep}
  \end{equation}
where $\om_{B_Z,\ep}$ is an approximate conical metric along $B_Z$, cf. e.g. \cite[Sect.~3]{GP}.\\

  \noindent
  \textbf{Step 2. Bounding the Ricci curvature from below}
  
  \noindent
The heart of the proof relies on the following formula due to Royden, cf \cite[Prop.~9]{WY16}, valid on the Zariski open set $U\subset Z$ defined by $U:=\mu^{-1}(\Omega \cap \mathrm{Amp}(K_X+bB))$, and where $\tom:=\mu^*\om$.
\begin{equation}
\label{Royden}
\Delta_{\omte} \log \tr_{\omte}\tom \ge \kappa \cdotp \tr_{\omte}\tom-\lambda
\end{equation}
%$\kappa:U\to\mathbb R_+$ is a function such that $- \frac{n+1}{2n} \cdotp \kappa(z)$ is a nonpositive upper bound for the holomorphic sectional curvature of $\tom_z$ 
where $\kappa:=\frac{n+1}{2n} \cdotp \kappa_0$ and $\lambda: Z\to \mathbb R_+$ is any function such that $\Ric \omte \ge -\lambda \omte$.

The first step is to get an explicit expression for $\lambda$, and then to write a global regularized version of \eqref{Royden} that we could integrate over the whole $Z$. \\

Keeping in mind that we want to get a lower bound of $\Ric \omte$, it is clear from \eqref{ric} that $\theta_{\ep}^B$ and $\theta_{\ep}^E$ will not play the same role. We first deal with the easier term
 \begin{align*}
 \theta_{\ep}^B &= \sum_{i=0}^db_i \left( \frac{\ep^2 }{(|s_i|^2+\ep^2)^2}\cdotp \,\la Ds_i,Ds_i\ra+ \frac{\ep^2}{|s_i|^2+\ep^2} \cdotp \Theta_{h_{B_i}}\right) \\
 & \ge -f_{\ep}^B \, \om_Z
 \end{align*}
 where $f_{\ep}^B:= C \left (\sum_{i=0}^d \frac{\ep^2}{|s_i|^2+\ep^2} \right)$ for some $C>0$ large enough. In particular, one gets 
 \begin{equation}
 \label{B}
  \theta_{\ep}^B \ge -(f_{\ep}^B\, \tr_{\omte}\om_Z) \cdotp \omte
 \end{equation}
 Similarly, one can decompose $$\theta^E_{\ep}= \ae+\be$$ where $\ae \ge 0$ and $\pm \be \le C \left(\sum_j \frac{\ep^2}{|t_j|^2+\ep^2}\right) \cdotp \omz$ for some uniform constant $C>0$.  More precisely, $\ae= \sum_j a_j\frac{\ep^2}{(|t_j|^2+\ep^2)^2} \cdotp \langle Dt_j, Dt_j \rangle$ and $\be= \sum_ja_j \frac{\ep^2 }{|t_j|^2+\ep^2}\cdotp \Theta_{h_{E_j}}. $ If we define $f_{\ep}^E:=C \left(\sum_j \frac{\ep^2}{|t_j|^2+\ep^2}\right) $ for some large $C>0$, then we have 
\begin{align*}
\theta^E_{\ep} &\le \ae+f_{\ep}^E \,\omz\\
& \le \tr_{\omte}(\ae+f_{\ep}^E \,\omz) \, \cdotp \omte \\
& = \tr_{\omte} (\theta^E_{\ep} +(f_{\ep}^E \omz-\be)) \cdotp \omte \\
& \le \tr_{\omte}( \theta^E_{\ep}+ 2f_{\ep}^E \,\omz) \cdotp \omte
\end{align*}
Let us now set $\chi_{\ep}:=f_{\ep}^B+2f_{\ep}^E$ \label{chi}; this is a smooth, positive function bounded uniformly when $\ep\to 0$ and such that $\chi_{\ep}\to 0$ almost everywhere. From \eqref{ric}, \eqref{B} and the inequality above, one deduces that
$$\Ric \omte \ge -\big(1+\tr_{\omte}(\theta_{\ep}^E+\chi_{\ep}\, \om_Z)\big) \cdotp \omte$$
which, along with \eqref{Royden}, yields the following formula valid on $U$
\begin{equation}
 \label{Royden2}
\Delta_{\omte} \log \tr_{\omte}\tom \ge \kappa \cdotp \tr_{\omte}\tom-\tr_{\omte}( \theta_{\ep}^E+\chi_{\ep}\, \omz) -1
\end{equation}
\bigskip

\noindent
\textbf{Step 3. Integration by parts. }

\noindent 
Because $\tom$ might vanish outside of $U$, the left-hand side of \eqref{Royden2} might become singular across $Z\smallsetminus U$. So let us choose $\delta>0$; it is easy to deduce from \eqref{Royden2} the following inequality
\begin{equation}
 \label{Royden3}
 \Delta_{\omte} \log (u+\delta) \ge  \kappa \cdotp \frac{u^2}{u+\delta}-v\cdotp \frac{u}{u+\delta}
 \end{equation}
where $u:=\tr_{\omte}\tom$ and $v=\tr_{\omte}( \theta_{\ep}^E+ \chi_{\ep} \,\omz)+1 $ are smooth, \textit{nonnegative} functions on the whole $Z$ which depend on $t,\ep>0$. Indeed, the inequality \eqref{Royden2} can be rewritten as 
$$\Delta u \ge \kappa u^2+\frac 1 u |\nabla u |^2-vu$$
hence 
\begin{align*}
\Delta \log(u+\delta) & \ge \frac{1}{u+\delta}  \cdotp (\kappa u^2+\frac 1 u |\nabla u |^2-vu) -\frac{1}{(u+\delta)^2} \cdotp |\nabla u |^2 \\
& = \kappa \cdotp \frac{u^2}{u+\delta}-v\cdotp \frac{u}{u+\delta} + \left( \frac{1}{u(u+\delta)}- \frac{1}{(u+\delta)^2} \right) \cdotp |\nabla u |^2 \\
& \ge \kappa \cdotp \frac{u^2}{u+\delta}-v\cdotp \frac{u}{u+\delta} 
\end{align*}
and \eqref{Royden3} follows. \\

As both sides of \eqref{Royden3} are continuous on $Z$ (remember that $t,\ep, \delta>0$ are fixed for the time being), the inequality extends across $Z\setminus U$. Then, one can multiply each side by $\omte^m$ and integrate over $Z$. We get
\begin{equation*}
  \int_Z  \kappa \cdotp \frac{u^2}{u+\delta} \,\omte^m \le \int_Zv\cdotp \frac{u}{u+\delta} \, \omte^m
 \end{equation*}
 By dominated convergence, one can pass to the limit in the integrals when $\delta \to 0$ to get
 \begin{equation}
 \label{Royden4}
  \int_Z  \kappa \cdotp \tr_{\omte}\tom \,\cdotp \omte^m \le \int_Z(\tr_{\omte}( \theta^E_{\ep}+ \chi_{\ep} \,\omz)+1) \, \omte^m
 \end{equation}
 
 \bigskip 
 
\noindent
\textbf{Step 4. Computing the error terms}

\noindent 
Let us now analyze the right-hand side of \eqref{Royden4}, which coincides with
\begin{equation}
\label{rhs}
m\int_{Z}\theta^E_{\ep} \wedge \omte^{m-1}+m\int_Z\chi_{\ep}\, \omz\wedge \omte^{m-1}+\{\nu^*\om_A+t\om_Z\}^m
\end{equation}
The first and last terms of \eqref{rhs} are cohomological. The first term is equal to $$m \, E_Z\cdotp(\nu^*A+t\{\omz\})^{m-1} =mt^{m-1} E_Z\cdotp \{\omz\}^{m-1}$$ as $E_Z$ is $\nu$-exceptional, hence it converges to zero when $t\to 0$. The last one converges to $(A^m)=\mathrm{vol}(K_X+bB)$ when $t\to 0$. 
\smallskip
\noindent
As for the second term, it can be estimated at $t>0$ fixed thanks to \eqref{conic} by the integral 
\begin{equation*}
C_t \int_Z \chi_{\ep} \, \om_{B_Z}^m
\end{equation*}
where $\om_{B_Z}$ is a metric with conical singularities along $B_Z$. In particular, $\om_{B_Z}^m=g \om_Z^m$ for some density $g\in L^1(\om_Z^m)$. As $\chi_{\ep}$ is uniformly bounded and tends to $0$ almost everywhere when $\ep$ approaches $0$, 
the dominated convergence theorem asserts that $$\lim_{\ep\to 0} \int_Z\chi_{\ep}\, \omz\wedge \omte^{m-1} = \lim_{\ep\to 0} \int_Z \chi_{\ep} \, \om_{B_Z}^m=0.$$ 
 In conclusion, one gets
 \begin{equation}
 \limsup_{t\to 0} \limsup_{\ep \to 0}   \int_Z  \kappa \cdotp \tr_{\omte}\tom \,\cdotp \omte^m \le \mathrm{vol}(K_X+bB) 
 \end{equation}

 \bigskip 
 
\noindent
\textbf{Step 5. Conclusion}

\noindent
Let us fix a relatively compact open set $K \Subset \mathrm{Amp}(K_X+bB)\backslash B$. Given \eqref{keb}, we know that on $\mu^{-1}(K)$, $\mu^*\omke$ is the smooth limit of $\omte$ when $t,\ep$ approach zero. Therefore
\begin{align*}
 \int_{K}  \kappa \cdotp \tr_{\omkeb}\om \,\cdotp \omkeb^m  &=  \int_{\mu^{-1}(K)}  \kappa \cdotp \tr_{\mu^*\omke}\tom \,\cdotp (\mu^*\omke)^m \\
&= \limsup_{t\to 0} \limsup_{\ep \to 0}   \int_{\mu^{-1}(K)}   \kappa \cdotp \tr_{\omte}\tom \,\cdotp \omte^m\\
 & \le \limsup_{t\to 0} \limsup_{\ep \to 0}   \int_Z  \kappa \cdotp \tr_{\omte}\tom \,\cdotp \omte^m\\
 &  \le \mathrm{vol}(K_X+bB) 
\end{align*}
and as this holds for any $K$, we get the desired inequality. Proposition~\ref{intcontrol} is proved.
\end{proof}

\subsection{End of the proof}
This section is devoted to the proof of Corollary~\ref{coro}.

\begin{proof}[Proof of Corollary~\ref{coro}]
\label{pcoro}
We first claim that $K_X$ is pseudoeffective. Indeed, observe that if $f:\mathbb P^1\to X$ is a rational curve whose image hits $\Omega$, then there exists a finite set $\Sigma \subset \mathbb P^1$ such that $f(\mathbb P^1\backslash \Sigma) \subseteq \Omega$. Then one can apply the inequality \cite[Prop.~4]{Royden} to $f:(\mathbb P^1\backslash \Sigma, \om_{\rm FS}) \to (\Omega, \om)$ to get 
$$\Delta_{\om_{\rm FS}} \log \tr_{\om_{\rm FS}}(f^*\om) \ge \kappa \,  \tr_{\om_{\rm FS}}(f^*\om)+2$$
where $\kappa:=\frac{m+1}{m}\cdotp \kappa_0$. In particular, the function $ \log \tr_{\om_{\rm FS}}(f^*\om) $ on $\mathbb P^1\backslash \Sigma$ is subharmonic and bounded above. Therefore it extends to a subharmonic function on $\mathbb P^1$, hence it has to be constant which is a contradiction. This shows that every rational curve on $X$ is contained in the Zariski closed proper subset $X\backslash \Omega$, hence $K_X$ is pseudoeffective by \cite{BDPP}. Note that we only used the boundedness from above of $\tr_{\om_{\rm FS}}(f^*\om) $ near the complement of $\Omega$ and not its smoothness across $X\backslash \Omega$. This will be useful later, cf Step~6 on page \pageref{Step6}. \\

Let $B$ be an ample divisor on $X$. For any rational number $b>0$, the $\Q$-line bundle $K_X+bB$ is big, hence there exists a unique Kähler-Einstein metric $\omb\in c_1(K_X+bB)$ on $X$ solving 
$$\Ric \omb = -\omb +b[B]$$
cf \cite{BEGZ} or \cite[Thm.~2.2]{G2}. In terms of Monge-Ampère equation, if $\theta$ (resp. $\theta_B$) is a smooth representative of $c_1(K_X)$ (resp. $c_1(B)$), then $\omb=\theta+b\theta_B+\ddc \vpb$ solves 
$$\la(\theta+b\theta_B+\ddc \vpb)^m\ra=\frac{e^{\vpb}}{|s|^{2b}}\,dV$$
where $dV$ is a fixed smooth volume form such that $\Ric dV=-\theta$, $s$ is a section of $\mathcal O_X(B)$ cutting out $B$ and $|\cdotp|$ is a smooth hermitian metric on $\mathcal O_X(B)$ whose curvature is equal to $\theta_B$. Thanks to \textit{loc. cit.}, $\omb$ has full mass; that is  $$\int_{X}\la \omb^m\ra = \mathrm{vol}(K_X+bB)$$
and, moreover, $\omb$ is a genuine smooth Kähler-Einstein metric on the Zariski open set $\Omega_b:=\mathrm{Amp}(K_X+bB) \backslash B$. Combining this with the content of the Proposition, one gets a uniform constant $C>0$ such that the following inequality holds
\begin{equation}
\label{intt}
\int_{\Omega_b} \tr_{\omb} \om \, \cdotp \omb^m \le C\, \mathrm{vol}(K_X+bB)
\end{equation}
Les us define $M_b:=e^{\sup_X \vpb}$ and $u_b:=\vpb- \sup_X \vpb$ so that $(u_b)_{b>0}$ is a family of sup-normalized $C\om_X$-psh functions for some $C>0$ large enough, independent of $b$. In particular, $(u_b)_{b>0}$ is relatively compact in $L^1(dV)$, hence by the dominated convergence theorem, there exists $C>0$ independent of $b\in (0,1/2)$ such that
$$C^{-1} \le \int_X \frac{e^{u_b}}{|s|^{2b}}dV \le C$$
hence
\begin{equation}
\label{vol}
\mathrm{vol}(K_X+bB) = M_b \int_X \frac{e^{u_b}}{|s|^{2b}}dV \in [C^{-1} M_b, CM_b] 
\end{equation}
In particular, \eqref{intt} allows us to conclude that
\begin{equation}
\label{inttt}
\int_{\Omega_b} \tr_{\omb} \om \, \cdotp \omb^m \le CM_b
\end{equation}
On $\Omega_b$, one has the following standard inequality
\begin{align*}
 \tr_{\omb}\om  &\ge \left(\frac{\om^m}{\omb^m}\right)^{1/m} \\
 & = (\om^m/dV)^{1/m} (M_be^{u_b}/|s|^{2b})^{-1/m} 
\end{align*}
Now let $K\Subset \Omega$ be a relatively compact open subset which is located away from the degeneracy locus of $\om$ so that $(\om^m/dV)^{1/m} \ge C^{-1}>0$ on $K$, up to taking $C$ larger. Then, one has
\begin{align*}
 \int_{K\cap \Omega_b }\tr_{\omb}\om \, \omb^m   & \ge C^{-1} M_b^{1-1/m} \, \int_{K \cap \Omega_b} e^{(1-\frac 1m) u_b} |s|^{2b\left(\frac 1m -1\right)} dV \\
  & = C^{-1} M_b^{1-1/m} \, \int_{K} e^{(1-\frac 1m) u_b} |s|^{2b\left(\frac 1m -1\right)} dV \\
 & \ge {C'}^{-1} M_b^{1-1/m}
 \end{align*}
for some $C'>0$ independent of $b$ as $K\backslash (K\cap \Omega_b)$ has zero Lebesgue measure. Combined with \eqref{inttt} one gets that 
$$M_b \ge C^{-1}M_b^{1-1/m}$$
for some uniform $C>0$. In particular, $M_b$ is uniformly bounded from below away from zero, hence one deduces from \eqref{vol} the existence of $\eta>0$ independent of $b>0$ such that 
$$\mathrm{vol}(K_X+bB)>\eta.$$
By the continuity of the volume function, cf \cite[Thm.~2.2.37]{PAG1}, one deduces that $\mathrm{vol}(K_X)>0$, hence $K_X$ is big and $X$ is of general type.
\end{proof}

Let us finish this section with the following

\begin{rema}
 \label{lt}
 In the setting of the Theorem A, one can additionally see that if $Y$ has log terminal singularities (see e.g. \cite[Def.~2.34]{KM} for a definition), then $K_Y$ is an ample $\Q$-line bundle. 

\noindent
To see this, first observe that $K_Y$ is a big $\Q$-line bundle because a desingularization $\widetilde Y $ of $Y$ is of general type and there is a natural inclusion $H^0(\widetilde Y, mK_{\widetilde Y}) \subseteq H^0(Y,mK_Y)$ for any integer $m$ divisible enough. By \cite[Thm A.(ii)]{BBP}, the augmented base locus of $K_Y$ is uniruled, hence empty, as $M$ does not contain any rational curve. 
\end{rema}

\section{The quasi-negative case}
\label{quasin}
The argument in the proof of Theorem~A is relatively robust and allows us to work with a weaker assumption on the holomorphic section curvature of $(M,\om)$. More precisely, let us consider a compact Kähler manifold $(M,\om)$ with \textit{quasi-negative} holomorphic sectional curvature; that is
\begin{enumerate}
\item[$(i)$] For any pair $(x,[v])\in M \times \mathbb P(T_{M,x})$, one has $\HSC_{\om}(x,[v]) \le 0$.
\item[$(ii)$] There exists $x_0\in M$ such that for any $[v] \in \mathbb P(T_{M,x_0})$, one has $\HSC_{\om}(x_0,[v]) <0$.
\end{enumerate} 
In this setting, Diverio-Trapani \cite{DT16} proved that the conclusions of \cite{WY16, TosYang} hold as well, namely $M$ is projective and $K_M$ is ample. Introducing the (open) negative curvature locus $$\mathcal W:=\{x\in M;  \forall v \in T_{M,x}\backslash \{0\}, \HSC_{\om}(x,[v]) <0\}$$
one can use again the decreasing property of the holomorphic bisectional curvature to conclude that any \textit{smooth} subvariety $Y\subset M$ such that $Y \cap \mathcal W \neq \emptyset$ satisfies that $K_Y$ is ample.  The goal of this section is to extend this result to singular subvarieties:

\begin{theo}
\label{thmqn}
Let $(M,\om)$ be a compact Kähler manifold with quasi-negative holomorphic sectional curvature. Let $Y$ be a possibly singular irreducible subvariety $Y\subset M$ such that $Y \cap \mathcal W \neq \emptyset$. Then $Y$ is of general type.  
\end{theo}

The proof of Theorem~\ref{thmqn} is very much similar to the proof of Theorem~A. Considering a resolution of the normalization of $Y$, one gets a smooth projective manifold $X$ which is not uniruled as $M$ contains no rational curve. Again, using \cite{BDPP}, $K_X$ is pseudo-effective. Then one considers an ample line bundle $B$ on $X$ and a rational number $b>0$ so that $K_X+bB$ is big, hence there is a unique KE metric $\omb \in c_1(K_X+bB)$. The pull-back of the Kähler metric on $X$ will still be denoted $\om$, as in the case of Theorem~A. Let us point out the main adjustments that need to be performed in the quasi-negative case. \\

\noindent
\textbf{Step 1.}

\noindent
There is no change to be made here, as we consider the same metrics $\omte$ on $Z$. \\

\noindent
\textbf{Step 2.}

\noindent
On the Zariski open set $U:=\mu^{-1}(\Omega \cap \mathrm{Amp}(K_X+bB))\subset Z$, the Laplacian inequality now becomes
\begin{equation}
\label{Rroyden}
\Delta_{\omte} \log \tr_{\omte}\tom \ge \kappa \cdotp \tr_{\omte}\tom-\lambda
\end{equation}
where $\kappa:U\to\mathbb R_+$ is a function such that $- \frac{n+1}{2n} \cdotp \kappa(z)$ is a nonpositive upper bound for the holomorphic sectional curvature of $\tom_z$ and $\lambda: Z\to \mathbb R_+$ is a function such that $\Ric \omte \ge -\lambda \omte$, as before. 

\noindent 
The continuous function $\kappa:U\to \mathbb R_+$ does not necessarily extend to a continuous function on $Z$. However it easy to construct a continuous function $\widetilde \kappa:Z\to \mathbb R_+$ along with two small neighborhoods $W\subset W'$ of $Z\smallsetminus U$ with the following properties
\begin{enumerate}
\item[$\cdotp$] $\widetilde \kappa|_{W}\equiv 0$
\item[$\cdotp$] $\widetilde \kappa =  \kappa$ on $U\smallsetminus W'$
\item[$\cdotp$] $\widetilde \kappa \le \kappa$ on $U$
\item[$\cdotp$] $(U\smallsetminus W') \cap (p\circ \mu)^{-1}(Y \cap \mathcal W) \neq \emptyset$
\end{enumerate}
Because of the third point, the formula \eqref{Royden3} remains true if one replaces $\kappa$ by $\widetilde \kappa$. \\

\noindent
\textbf{Steps 3-5.}

\noindent
No change is needed here. The conclusion we get is

\begin{equation}
\label{qn}
\int_{\mathrm{Amp}(K_X+bB)\backslash B} \widetilde \kappa \cdotp \tr_{\omb}\om \, \cdotp  \omb^m\,  \le \, \mathrm{vol}(K_X+bB).
\end{equation}

Moving on to the last part of the proof, one can pick a relatively compact subset $K \Subset \Omega \cap  p^{-1}(\mathcal W) \subset X$ such that on $K$, one has $\widetilde \kappa \ge C^{-1}$ and $(\om^m/dV)^{1/m} \ge C^{-1}$. Therefore
\begin{align*}
\int_{K \cap \Omega_b} \tr_{\omb} \om \, \cdotp \omb^m &\le C\int_{K \cap \Omega_b}\widetilde \kappa \cdotp \tr_{\omb} \om \, \cdotp \omb^m \\
& \le  C\, \mathrm{vol}(K_X+bB)
\end{align*}
At this point, the same arguments as before show that $\mathrm{vol}(K_X+bB) \in [C^{-1} M_b, CM_b] $ as well as $$ \int_{K\cap \Omega_b }\tr_{\omb}\om \, \omb^m   \ge {C}^{-1} M_b^{1-1/m}$$
from which the uniform positive lower bound on $\mathrm{vol}(K_X+bB)$ follows.

\section{The quasi-projective case}
\subsection{Setting}
 \label{setting2}
 Let $X$ be a smooth, complex projective variety of dimension $m$ and let $D=\sum_{k=0}^p D_i$ be a reduced divisor with simple normal crossings. Let $\X:=X\backslash D$ and let $\om$ be a Kähler form on $\X$ such that there exists $\kappa_0 >0$ such that  $$\forall \, (x,[v])\in \X\times \mathbb P(T_{X,x}), \quad \HSC_{\om}(x,[v])\le -\kappa_0.$$
 
 \noindent
 In this setting, one can deduce from the Ahlfors-Schwarz lemma the following 
  \begin{lemm}
 \label{AS}
 In the setting \ref{setting2} above, the following statements hold
 \begin{enumerate}
 \item Every holomorphic map $f:\mathbb C \to \X$ is constant. 
 \item The Kähler metric $\om$ is dominated by a Kähler metric $\omega_{\rm P}$ on $\X$ with Poincaré singularities along $D$. 
 \end{enumerate}
 \end{lemm}

 \noindent
 Recall that a Kähler metric $\omp$ on $\X$ is said to have Poincaré singularities along $D$ if for any $x\in D$ and any coordinate chart $U\simeq \Delta^m$ around $x$ where $D$ is given by $(z_1\cdots z_r=0)$, $\om|_{|U}$ is quasi-isometric to the model Poincaré metric 
 $$\omega_{\rm mod} := \sum_{k=1}^r \frac{i\,dz_k\wedge d\bar z_k}{|z_k|^2 \log^2 |z_k|^2}+\sum_{k=r+1}^m i\, dz_k\wedge d\bar z_k$$
 
 \begin{proof}[Proof of Lemma~\ref{AS}]
 The first item is a consequence of \cite[Cor.~1]{Royden}. The second one is a consequence of \cite[Thm.~1]{Royden} applied to $f=\mathrm{id}:((\Delta^*)^r\times \Delta^{m-r}, \omega_{\rm mod}) \to ((\Delta^*)^r\times \Delta^{m-r}, \omega)$ where one identifies $U\cap \X$ with $(\Delta^*)^r\times \Delta^{m-r}$. 
 \end{proof}
 
 \begin{rema}
 At this point, one would like to conclude that $K_X+D$ is pseudoeffective. Indeed, if $K_X+D$ were to fail to be pseudo-effective, then by \cite{BDPP}, one would obtain a covering family of curves $(C_t)$ such that $(K_X+D)\cdotp C_t <0$. Following Mori's bend and break, one could deform each curve $C_t$ into a new reducible curve containing a rational curve $C_t'$ passing through a given point. From the second item of Lemma~\ref{AS}, one would obtain a contradiction if one knew that $C_t'$ intersects $D$ in a most two points. Therefore, the pseudoeffectiveness of $K_X+D$ would be a consequence of the following general conjecture of Keel-McKernan
 
  \begin{conj}
  \label{LBB}
   \emph{ \textbf{(Logarithmic bend and break, cf {\cite[1.11]{KeelMcK}})} }
  
  \noindent
Let $(X,D)$ be a pair consisting of a smooth projective complex variety $X$ and a reduced divisor $D$ with simple normal crossings.

\noindent
 If $C\subset X$ is a curve such that $(K_X+D)\cdotp C<0$ and $C\nsubseteq D$, then through a general point of $C$ there is a rational curve meeting $D$ at most once. 
 \end{conj}
 
The logarithmic bend and break is known in dimension two by \cite[1.12]{KeelMcK}.  Let us also mention that Lu-Zhang \cite[Thm.~1.4]{LZ} and McQuillan-Pacienza \cite[Rem.~1.1]{McQP} proved the above conjecture assuming that for any non-empty subset $ J \subset I$, any holomorphic map $f:\mathbb C \to \bigcap_{j\in J} D_i \backslash \bigcap_{k\notin J} D_k$ is constant. 
 \end{rema}

 \subsection{The main statement}
 Let $B$ be a smooth divisor on $X$ such that 
 \begin{enumerate}
 \item $B+D$ has simple normal crossings
 \item The line bundles associated to $B$ and $B-D$ are ample
 \item There exists $s_0 \in (0,\frac 12)$ such that $K_X+D+s_0B$ is pseudo-effective. 
 \end{enumerate}
Now, let $0 \le s< 1/2$ be any rational number such that the $\Q$-line bundle $K_X+D+sB$ is pseudoeffective. Up until the very end, the number $s$ will be fixed. By the assumptions on $B$ above, one knows that for any rational number $b>0$, the $\Q$-line bundle $K_X+D+sB+b(B-D)=K_X+(1-b)D+(b+s)B$ is big. Let 
$$\Db:=(1-b)D+(b+s)B$$ 
The pair $(X,\Db)$ is klt and is of log general type whenever $b\in (0, 1/2)$. By \cite{EGZ}, $(X,\Db)$ admits a unique Kähler-Einstein metric $\omkebs$. That is, $\omkebs$ is a closed, positive current in $c_1(K_X+\Db)$ with bounded potentials satisfying the Einstein equation 
$$\Ric \omkebs = -\omkebs +[\Db]$$
in the weak sense. Thanks to \cite{BCHM}, $\omkebs$ defines a smooth Kähler metric on the Zariski open set $\mathrm{Amp}(K_X+\Db)\backslash (D\cup B)$. 
 
Indeed, thanks to \textit{loc. cit.}, there exists a canonical model $(X_{{\rm can},b,s},\Delta_{{\rm can},b,s})$ of $(X,\Db)$ with klt singularities such that $K_{X_{{\rm can},b,s}}+\Delta_{{\rm can},b,s}$ is ample. Let us consider a resolution $Z$ of the graph of the birational map $\phi:X\dashrightarrow X_{{\rm can},b,s}$ as summarized in the diagram below
 $$
  \xymatrix{
    & \ar[ld]_{\mu} Z  \ar[rd]^{\nu}   &\\
      X  \ar@{-->}[rr]_{\phi} &  &   X_{{\rm can},b,s} 
  }
  $$
 Let $\Db':=(1-b)D'+(b+s)B'$ where $D'$ (resp. $B'$) is the strict transform of $D$ (resp. $B$) by $\mu$. There exist a $\nu$-exceptional $\mathbb Q$-divisor $E:=\sum_{j=0}^d a_j E_j$ with snc support and coefficients $a_j\in (-1,+\infty)$ such that 
   $$K_Z+\Db'=\nu^* (K_{X_{{\rm can},b,s}}+\Delta_{{\rm can},b,s})+E$$
Moreover, one can assume that $\mathrm{Exc}(\mu)$ is divisorial and that the support of $\Db'+E$ has simple normal crossings. Up to setting some $a_j$'s to zero, one can also assume that $\mathrm{Exc}(\mu) \subseteq \bigcup_{j=0}^d E_j$. Let us stress here that $\mu$ is an isomorphism over the Zariski open set $\mathrm{Amp}(K_X+\Db)$ given that $\phi$ is defined there and induces an isomorphism onto its image when restricted to that set. \\
  
  Let $A:=K_{X_{{\rm can},b,s}}+\Delta_{{\rm can},b,s}$ and let $\om_Z$ be a background Kähler metric on $Z$. For any $t\in [0,1]$, the cohomology class $c_1(\nu^*A+t\{\om_Z\})$ is semi-positive and big (it is even Kähler if $t>0$). Thus,  it follows from \cite{EGZ} that there exists a unique singular Kähler-Einstein metric $\om_t\in c_1(\nu^*A+t\{\om_Z\})$ solving 
  $$\Ric \omt = -\omt + t\om_Z+[\Db']-[E]$$
  The current $\omt$ is smooth outside $\Supp(\Db'+E)$ and, moreover, there exists an effective, $\mu$-exceptional $\Q$-divisor $F$ on $Z$ such that 
  \begin{equation}
  \label{keb}
  \mu^*\omkebs = \om_0+[F].
  \end{equation} 
  In particular, $\omkebs$ is smooth on $\mathrm{Amp}(K_X+\Db)\backslash (D\cup B)$. \\

\noindent
As in the earlier setting, the key point is the following volume estimate
\begin{theo}
\label{intcontrol2}
In the setting \ref{setting} above, given an ample line bundle $H$ on $X$, there exists a constant $C$ depending only on $X,D,H,\om$ \---but not $b$ or $s$\--- such that 
$$\int_{\mathrm{Amp}(K_X+\Db)\backslash (D\cup B)} \tr_{\omkebs}\om \, \cdotp  \omkebs^m\,  \le C \left( \la (K_X+\Db)^m\ra + b \la (K_X+\Db)^{m-1}\cdotp H\ra\right)$$
where $\la \cdotp \ra$ is the movable intersection product, cf \cite[\S 3]{BDPP} and the references therein. Furthermore, the line bundle $K_X+D$ is big.
\end{theo}

\subsection{Proof of Theorem~\ref{intcontrol2}}
The strategy of the proof is similar to that of Proposition~\ref{intcontrol}, but it gets more technical. We will only indicate what are the main changes to perform.\\

  \noindent
  \textbf{Step 1. }
  
  \noindent
 For $t, \ep>0$ we now instead consider the current $\omte=\nu^*\om_A+t\om_Z+dd^c \vpte \in c_1(\nu^*A+t\{\om_Z\})$ solving 
  
  \begin{equation}
  \label{MA3}
  (\nu^*\om_A+t\om_Z+dd^c \vpte)^m=\frac{\prod_{j=0}^d (|t_j|^2+\ep^2)^{a_j}} {|s_{D'}|^{2(1-b)}\cdotp |s_{B'}|^{2(b+s)}}\cdotp e^{\vpte} dV
  \end{equation}
  where $t_j,s_{D'},s_{B'}$ are respectively sections of $\mathcal O_Z(E_j),\mathcal O_Z(D'),\mathcal O_Z(B')$ cutting out $E_j,D',B'$ and the smooth hermitian metrics chosen on the various bundles are such that the following equation  holds.
  \begin{equation}
  \label{ric3}
  \Ric \omte = -\omte + t\om_Z+ [\Db']-\theta_{\ep}^E
  \end{equation}
  where
$\theta_{\ep}^E:=\sum_{j=0}^d a_j (\Theta_{h_{E_j}}+dd^c \log(|t_j|^2+\ep^2))$.  
  \smallskip
  
  \noindent
  In the following, one sets $\Zsd:=Z\backslash (D'\cup B')$. By the proof of \cite[Prop.~2.1]{GP}, $\omte$ is a Kähler metric on $\Zsd$ with conical singularities along $\Db'$, and it is uniformly (in $\ep$) dominated by a Kähler metric on $\Zsd \backslash \cup_{a_j<0} E_j$ with conic singularities along $\Db'+\sum_{a_j<0}(-a_j)E_j$. In particular, there exists $f\in L^1(\omz^m)$ independent of $\ep$ such that $$\omz \wedge \omte^{m-1} \le f\omz^m$$ By Lebesgue dominated convergence theorem, one gets  
   \begin{equation}
 \label{convint}
 \forall t>0, \forall j=0 \ldots d, \quad \lim_{\ep\to 0} \int_{\Zsd} \frac{\ep^2}{|s_{E_j}|^2+\ep^2} \, \om_Z \wedge \omte^{m-1} = 0.
 \end{equation}
  
  \noindent
 Finally, remember from Ahlfors-Schwarz lemma, cf Lemma~\ref{AS}, that the Kähler metric $\tom:=(\mu|_{\Zsd})^*\om$ on $\Zsd$ has at most Poincaré singularities along $D'+\sum E_j$. In particular, one has 
  \begin{equation}
  \label{borne}
  \sup_{\Zsd} \left[|s_{D'}|^{2b}\cdotp \prod_{j=0}^d |s_{E_j}|^2 \cdotp \tr_{\omte} \tom \right] <+\infty.
  \end{equation}

 %  The addition of the parameter $\delta$ is to make sure that $(\mu|_{\X})^*\om$ is dominated by some multiple of $\omted$ in order to perform some integrations by parts, cf next steps. As soon as these will have been performed, one will set $\delta=0$, $\omte:=\omega_{t,\ep, 0}$. In this new setting, it is unclear wether the suitable equivalent of \eqref{conic} holds for $\omte$. However, it was proved in \cite[Lem.~3.7]{GSS} that for any section $s=s_i$ or $s=t_j$ and each $t>0$ fixed one has

  \noindent
  \textbf{Step 2. } 
  
  \noindent
  The following Laplacian inequality holds on $\Zsd$
  \begin{equation}
 \label{RoydenQP}
\Delta_{\omte} \log \tr_{\omted}\tom \ge \kappa \cdotp \tr_{\omte}\tom-\tr_{\omte}( \theta_{\ep}^E + \chi_{\ep}\, \omz) -1
\end{equation}
where $\chi_{\ep} = C \sum_{j=0}^d \frac{\ep^2}{|s_{E_j}|^2+\ep^2}$ for some large $C$ independent of $\ep$. Moreover, one has
\begin{align*}
\Delta_{\omte}\left[\log|s_D'|^{2b}+\sum_{j=0}^d \log |s_{E_j}|^2 \right]& = \tr_{\omte}(dd^c (\log|s_D'|^{2b}+\sum_{j=0}^d \log |s_{E_j}|^2))\\
& \ge -b \cdot \tr_{\omte}\Theta_{D'}- \sum_{j=0}^d\tr_{\omte} \Theta_{E_j}
\end{align*}
where $\Theta_{D'}, \Theta_{E_j}$ are the Chern curvature form of the smooth hermitian metrics chosen on the respective associated line bundles. In the end, one gets the following identity, holding on $\Zsd$
\begin{equation}
 \label{RoydenQPP}
\Delta_{\omte} \big[ \log \big(|s_{D'}|^{2b}\cdotp \prod_{j=0}^d |s_{E_j}|^2 \cdotp \tr_{\omted}\tom \big) \big] \ge \kappa \cdotp \tr_{\omte}\tom-\tr_{\omte}\big( \theta_{\ep}^E + \chi_{\ep}\, \omz+b  \Theta_{D'}+ \sum_{j=0}^d \Theta_{E_j}\big)-1
\end{equation}

\noindent
\textbf{Step 3.}

\noindent
As before, one starts by choosing $\delta>0$ and deduce from \eqref{RoydenQP} the following
$$ \Delta_{\omte} \log (u+\delta) \ge  \kappa \cdotp \frac{u^2}{u+\delta}-v\cdotp \frac{u}{u+\delta}$$
where $u:=|s_{D'}|^{2b}\cdotp \prod_{j=0}^d |s_{E_j}|^2 \cdotp \tr_{\omted}\tom$ and $v=\tr_{\omte}\big( \theta_{\ep}^E + \chi_{\ep}\, \omz+b  \Theta_{D'}+ \sum_{j=0}^d \Theta_{E_j}\big)+1$. By the observation \eqref{borne} above, all the terms involved are smooth on $\Zsd$ and globally bounded. In particular, the dominated convergence theorem shows that 
\begin{equation}
\label{dc}
\int_{\Zsd} (\kappa u-v)\, \omte^m= \lim_{\delta \to 0 } \int_{\Zsd} \Big(\kappa \cdotp \frac{u^2}{u+\delta}-v\cdotp \frac{u}{u+\delta}\Big) \, \omte^m
\end{equation}
Combining \eqref{dc} with Lemma~\ref{ipp} below, one eventually gets
\begin{equation}
\label{inegalite}
\int_{\Zsd} \kappa \, \tr_{\omte}\tom\, \omte^m \le  \int_{\Zsd} \Big(\tr_{\omte}\big( \theta_{\ep}^E + \chi_{\ep}\, \omz+b  \Theta_{D'}+ \sum_{j=0}^d \Theta_{E_j}\big)+1\Big) \, \omte^m
\end{equation}

\begin{lemm}
\label{ipp}
Let $f,g \in  L^{\infty}(\Zsd) \cap \mathscr C^{\infty}(\Zsd)$ such that $$\Delta_{\om_c} f \ge g \quad \mbox{on } \, \Zsd,$$
where $\om_c$ is some Kähler metric with conic singularities along $\Db'$. Then 
$$\int_{\Zsd}g\, \om_c^m \le 0.$$
\end{lemm}

\begin{proof}[Proof of Lemma~\ref{ipp}]
It is well-known that the complex codimension one set $D'\cup B' \subseteq Z$ admits a family of cut-off functions $(\xi_{\alpha})_{\alpha>0}$ such that $$\limsup_{\alpha \to 0} \sup_Z |dd^c \xi_{\alpha}|_{\omp} <+\infty$$ where $\omp$ is a metric with Poincaré singularities along $D'+B'$, cf e.g. \cite[Sect.~9]{CGP}. 

\noindent
By assumption, the function $g$ is integrable with respect to $\om_c^m$ and by dominated convergence, one has 
$$\int_{\Zsd} g \, \om_c^n = \lim_{\alpha \to 0} \int_Z \xi_{\alpha}g \, \om_c^m$$
But that last integral is dominated by 
\begin{align*}
\int_Z \xi_{\alpha} \cdotp  \Delta_{\om_c} f \, \om_c^{m}&= m\cdotp \int_Z f dd^c \xi_{\alpha}\wedge \om_c^{m-1} \\
& \le C\cdotp \sup_Z |f| \cdotp \mathrm{Vol}_{\omp}(\Supp(\xi_{\alpha}))
\end{align*}
where $C$ is such $m\, dd^c \xi_{\alpha}\wedge \om_c^{m-1} \le C \omp^m$. Finally, the right-hand side tends to zero when $\alpha$ approaches zero. The Lemma is proved. 
\end{proof}

 \bigskip 
 
\noindent
\textbf{Step 4. }

\noindent 
The right-hand side of \eqref{inegalite} can be rewritten as
\begin{equation}
\label{rhs}
m\int_{Z}(\theta^E_{\ep} +b  \Theta_{D'}+ \sum_{j=0}^d \Theta_{E_j} )\wedge \omte^{m-1}+m\int_Z\chi_{\ep}\, \omz\wedge \omte^{m-1}+\{\nu^*\om_A+t\om_Z\}^m
\end{equation}
The first term is cohomological and coincides with $m\, \big((E+\sum E_j+bD')\cdotp (\{\nu^*\om_A+t\om_Z\})^{m-1}\big)$, which is independent of $\ep$. For the second, one has the limit computation \eqref{convint}. As $\sum E_j$ is $\nu$-exceptional, one gets
 \begin{equation}
 \label{sin}
 \limsup_{t\to 0} \limsup_{\ep \to 0}   \int_{\Zsd}  \kappa \cdotp \tr_{\omte}\tom \,\cdotp \omte^m \le mb(D'\cdotp (\nu^*A)^{m-1}) +\la(K_X+\Db)^m\ra
 \end{equation} 
 Finally, let $p>0$ such that $pH-D$ is effective. Then, one has  
 \begin{align*}
 D'\cdotp (\nu^*A)^{m-1}& \le(\mu^*D \cdotp (\nu^*A)^{m-1}) \\
 &\le p\,(\mu^*H\cdotp (\nu^*A)^{m-1})\\
 &= p\,\la H \cdotp (K_X+\Db)^{m-1} \ra 
\end{align*} 
which ends the proof of the first part of Theorem~\ref{intcontrol2}. 

 \bigskip 
 
\noindent
\textbf{Step 5. Bigness of $K_X+D$.}

\noindent 
As $\om$ is dominated by a metric with Poincaré singularities along $D$, Skoda-El Mir extension theorem implies that the current $\om$ on $\X$ can be extended to a closed, positive $(1,1)$-current on $X$ putting no mass on $D$. We still denote it by $\om$, and set $\alpha:=\{\om\}$; this is a pseudoeffective class. As $\om$ has no zero Lelong numbers, Demailly's regularization theorem shows that $\alpha$ is even nef, but we will not use this fact. We claim that for any $t>0$, one has
\begin{equation}
\label{mass}
\int_{\Zsd}  \tr_{\omte}\tom \,\cdotp \omte^m = m (\{\omte^{m-1}\}\cdotp \mu^*\alpha  )
\end{equation}
Indeed, the integral on the left-hand side can be rewritten as $m\int_{Z}  \tom \wedge \omte^{m-1}$ given that $\tom$ has at most Poincaré singularities. Moreover, for any $t,\ep>0$, the metric $\omte$ has conic singularities along $\Db'$ and can be regularized into a family of smooth Kähler metrics $(\om_{t,\ep, \delta})_{\delta>0}$ in the same cohomology class $\{\omte\}$ such that $\om_{t,\ep,\delta} \le C_{t,\ep} \omte$ for some $C_{t,\ep}>0$ independent of $\delta$. By Lebesgue dominated convergence theorem, one deduces that $$\int_{Z}  \tom \wedge \omte^{m-1} = \lim_{\delta \to 0} \int_Z \tom \wedge \omted^{m-1}. $$
Now, the total mass on $Z$ of a closed, positive $(1,1)$-current with respect to a given Kähler metric only depends on the cohomology class of that current. From the identity above, one deduces
 $$\int_{Z}  \tom \wedge \omte^{m-1} =\{\omte^{m-1}\} \cdotp \{\tom\}$$
 which prove \eqref{mass}.\\

%For any $\ep>0$, let $k(\ep)$ be the minimal integer such that $\psi_{k(\ep}}=\psi$ on $\{|s_D|\ge\ep/2\}$. As $\om$ is dominated by a Poincaré metric along $D$, one has
%\begin{align*}
%\la\om^m\ra = \int_{X\backslash D} \om^m \\
%& = \lim_{\ep \to 0} \int_{|s_D|\ge\ep} \om^m \\
%&= \lim_{\ep \to 0} \int_{|s_D|\ge\ep} (\theta+dd^c \psi_k)^m \quad (k\ge k(\ep))
%\end{align*}
%

%Let $\ep>0$; as $\om$ is dominated by a Poincaré metric along $D$, there exists a small open neighborhood $V_{\ep}$ of $D$ such that $\int_{X\backslash V_{\ep}} \om^m \ge \la \om^m \ra - \ep.$ As $\om$ is smooth on $X\backslash V_{\ep}$, there exists $k(\ep)$ such that $\psi_k = \psi$ on $X\backslash V_{\ep}$ for any $k \ge k(\ep)$. One deduces that for any $k\ge k(\ep)$, one has $\la\alpha^m\ra = \int_{X} (\theta+dd^c \psi_k)^m \ge \int_{X\backslash V_{\ep}} (\om+dd^c \psi_k)^m 

%For any $k>0$, one has $\la \alpha^m \ra = \int_X (\theta+dd^c \psi_k)^m
When $t,\ep$ approach zero, the right-hand side of \eqref{mass} converges to $m(  (\nu^*A)^{m-1}\cdotp\mu^*\alpha)$ which coincides with the movable intersection product $m\la (K_X+\Db)^{m-1} \cdotp \alpha \ra$. As a result, one obtains 
$$\la (K_X+\Db)^{m-1} \cdotp \alpha \ra \le \frac{bp}{\kappa}\la  (K_X+\Db)^{m-1} \cdotp H\ra+\frac 1{\kappa m} \la(K_X+\Db)^m\ra$$

\noindent
Let us now try to analyze the class $\alpha$. Because $\om$ is smooth and Kähler on a Zariski open set, $\alpha$ is big thanks to \cite{Bou02}. In particular, for $b$ small enough, one has an inequality of $(1,1)$ cohomology classes $$\alpha-\frac{bp}{\kappa} H \ge \frac 12 \alpha.$$ By the increasing and superadditive properties of the movable intersection \cite[Thm.~3.5 (ii)]{BDPP}, one has
$$\la (K_X+\Db)^{m-1} \cdotp \alpha \ra - \frac{bp}{\kappa}\la  (K_X+\Db)^{m-1} \cdotp H\ra \ge \frac 12 \la (K_X+\Db)^{m-1} \cdotp \alpha \ra$$
and therefore, using the Teissier-Hovanskii inequalities \cite[Thm.~3.5 (iii)]{BDPP}, one gets
\begin{align*}
\la(K_X+\Db)^m\ra &\ge \frac {\kappa m}2 \la (K_X+\Db)^{m-1} \cdotp \alpha \ra \\
& \ge \frac {\kappa m} 2  \la (K_X+\Db)^m\ra^{1-1/m}\cdotp \la\alpha^{m}\ra^{1/m} 
\end{align*}
%Clearly, one has $\mathbb{B}_+(\{\om\}) \subseteq D$. Moreover, as $\om$ is dominated by a metric with Poincaré singularities along $D$, $\om$ has zero Lelong numbers everywhere on $X$. Therefore, Demailly regularization theorem implies that $\{\om\}$ is nef. In particular, there exists an effective divisor $\wt D :=\sum d_i D_i$ supported on $D$ such that for any $b>0$ small enough, the class $\{\om\}-b\wt D$ is Kähler. Now, one has
%$$\la (K_X+\Db)^{m-1} \cdotp \{\om\} \ra \ge \la (K_X+\Db)^{m-1} \cdotp (\{\om\}-b\wt D) \ra+\la (K_X+\Db)^{m-1} \cdotp b \wt D \ra$$
or equivalently
$$\la(K_X+\Db)^m\ra \ge \Big(\frac {\kappa m}2 \Big)^m \cdot \la \alpha^m \ra$$
and the right-hand side is positive, independent of both $b$ and $s$. In conclusion, one gets 
\begin{equation}
\label{volu}
\mathrm{vol}(K_X+D+sB)=\lim_{b\to 0} \mathrm{vol}(K_X+\Db)\ge \Big(\frac {\kappa m}2 \Big)^m \cdot \la \alpha^m \ra.
\end{equation}
The inequality above holds for any rational number $s\ge 0$ such that $K_X+D+sB$ is pseudoeffective. If we can show that $K_X+D$ is pseudo-effective, then we are done as \eqref{volu} would show that $K_X+D$ is big. But if $K_X+D$ is not pseudoeffective, there exists a real number $s_{\infty}>0$ such that $K_X+D+s_{\infty}B$ is pseudoeffective but not big. Taking a sequence of rational numbers $(s_n)$ decreasing to $s_{\infty}$, one has that $K_X+D+s_nB$ is big with  $\mathrm{vol}(K_X+D+s_nB)\ge \Big(\frac {\kappa m}2 \Big)^m \cdot \la \alpha^m \ra$. By continuity of the volume function, one gets $\mathrm{vol}(K_X+D+s_{\infty}B)>0$ which is a contradiction.

\bigskip
\noindent
\textbf{Step 6. The case where $\om$ is bounded.}
\label{Step6}

\noindent 
Here the pseudoeffectivity of $K_X$ comes almost for free by the exact same argument as the one in the first step of the proof of Corollary~\ref{coro} (p.~\pageref{pcoro}) by setting $\Omega:=X\backslash D$. 

From there, one can reproduce almost \textit{verbatim} the arguments of the proof of Theorem A. The only difference is in Step 3. as the quantity $\tr_{\omte} \om$ is no longer smooth across $D$ but merely bounded. However, the integration by parts technique of Lemma~\ref{ipp} still applies as the family $(\xi_{\alpha}$ of cut-off functions satisfies $\pm dd^c \xi_{\alpha} \wedge \om_{\rm sm}^{m-1} \le \omp^m$ where  $\om_{\rm sm}$ is a smooth Kähler form on $X$ and $\omp$ is some Kähler form on $X\backslash D$ with Poincaré singularities along $D$. In particular, $\int_X |\Delta_{\om_{\rm sm}} \xi_{\alpha}| \, \om_{\rm sm}^m$ converges to $0$ as $\alpha$ approaches zero.

\backmatter

\bibliographystyle{smfalpha}
\bibliography{biblio}

\end{document}